\documentclass[12pt,a4paper]{amsart}

\usepackage{amsmath}
\usepackage{amsthm}
\usepackage{amssymb}

\newtheorem{corollary}{Corollary}[section]
\newtheorem{theorem}{Theorem}[section]
\newtheorem{lemma}{Lemma}[section]

\begin{document}
\title[On the the operator equation $AX-XB=C$]
{ On the operator equation $AX-XB= C$}
\author[S.Mecheri, A. Bachir, A. Segress ]{Salah Mecheri, Ahmed Bachir, Abdelkader Segress}
\address{Salah Mecheri\endgraf
  Mohamed El Bachir El Ibrahimi University \endgraf
  College MI, Department of Mathematics\endgraf
Bourdj Bou Arridj, Algeria}
 \email{mecherisalah@hotmail.com}
 \address{Ahmed Bachir\endgraf
  King Khaled University \endgraf
 College of Science, Department of Mathematics\endgraf
Abha, Saudi Arabia}
 \email{bishr@kku.edu.sa}
  \address{Abdelkadar Segress\endgraf
  Mascara University \endgraf
Department of Mathematics\endgraf
Mascara, Algeria}
 \email{bdelkader.segress@univ-mascara.dz}
\subjclass[2000]{Primary 47B47, 47A30, 47A63;\\
\hspace{4mm} Secondary 47A15.}
\begin{abstract}
This work studies how certain problems in quantum theory have
motivated some recent research in pure Mathematics in matrix and
operator theory. The mathematical key is that of a commutator
or a generalized commutator, that is, find an operator $X\in B(H)$
satisfying the operator equation $AX - XB = C$. By this we will
show how and why to solve the operator equation $AX - XB = C$. As application we recall
the solurion of  $AXB-CXD=E$ and we use the solution of this equation to solve
equation of Monkeypox Diseases.
\end{abstract}
 \maketitle
\section{ Introduction}
Let $B(H)$ be the algebra of all bounded linear operators on a separable
infinite dimensional complex Hilbert space H. This work studies how certain
problems in quantum theory have motivated some recent research in pure
Mathematics in matrix and operator theory. The mathematical key is that
of a commutator. Given $A, B \in B(H)$. The operator $C$ is said to be a
commutator, if there exists an operator $X \in B(H)$ such that $AX-XA = C$.
In general, if there exists an operator $X \in B(H)$ such that $AX - XB = C$, then $C$ is said to be a generalized commutator. The first important contribution to the study of commutators is due to A. Wintner who in 1947 proved that the identity operator I is not a commutator, that is, there are no element $X$ such that $I = AX - XA$ (1.1) (to see this, just take the trace of both sides of (1.1)). Nor can (1.1) hold for bounded linear operators $A$ and $X$: two nice proofs of this are due to Wielandt and A. Wintner [17]. Like much good mathematicians, Wintner's
theorem has its roots in physics. Indeed, it was prompted by the fact that
the unbounded linear maps P and Q representing the quantum-mechanical
momentum and position, respectively, satisfy the commutation relation
$$PQ - QP = ({{-ih}\over {2\pi}})I,$$
where $h$ is the Planck's constant and $I$ is the identity operator. Actually
one of the preoccupations is the structure of a commutator and a no commutator.

For this it is very interesting to solve the operator equation
$AX - XB = C$. In [14] W.E. Roth has shown for finite matrices $A$ and $B$
over a field that $AX - XB = C$ is solvable for $X$ if and only if the matrices
$$ \begin{pmatrix} A& 0 \\
                            0 & B \end{pmatrix}\, \text{and}\, \begin{pmatrix} A& B \\
                            0 & C \end{pmatrix}$$
are similar. A considerably briefer proof has been given by Flanders and
Wimmer [4]. In [13] Rosenblum showed that the result remains true when A
and B are bounded selfadjoint operators in $B(H)$. In this note we will generalize
these results for the case where $A$ is normal and ($A, B$) (resp ($B, A$))
satisfies $(FP)_{B(H)}$ (the Fuglede-Putnam property). Some open questions
are also given.
\section{Main Results}
Let $A, B \in B(H)$. We say that the pair $(A, B)$ satisfies $(FP)_{B(H)}$, if
$AC = CB$, where $C \in B(H)$ implies $A^{*}C = CB^{*}$ \cite{Mecheri1}.

In the following we will denote the spectrum, and the approximate spectrum
of an operator $A\in B(H)$ by $\sigma(A)$, and $\sigma_{a}(A)$ respectively.
\begin{lemma} If the matrix operator
$$\begin{pmatrix} Q& R \\
                            S & T\end{pmatrix}$$

defined on $H\oplus H$ is invertible, then the operator $S^{*}S + Q^{*}Q$
is invertible on $H$.
\end{lemma}
\begin{proof} Since $S^{*}S + Q^{*}Q$ is a positive operator,
$$\sigma(S^{*}S + Q^{*}Q) = \sigma_{a}(S^{*}S + Q^{*}Q).$$
If we assume that $S^{*}S + Q^{*}Q$ is not invertible, then there exists a sequence
$(x_{n})_{n}\subset H$ such that $||x_{n}||= 1$, for all $n\geq 1$ and $\lim_{n\rightarrow \infty}||(S^{*}S + Q^{*}Q)x_{n}|| = 0$. Consequently,
$$Lim_{n\rightarrow\infty} \parallel
\begin{pmatrix} Q& R \\
                            S & T\end{pmatrix} (x_{n}\oplus 0)\parallel^{2}=lim_{n\rightarrow \infty}\langle \begin{pmatrix} Q& R \\
                            S & T\end{pmatrix}^{*}\begin{pmatrix} Q& R \\
                            S & T\end{pmatrix}(x_{n}\oplus 0), (x_{n}\oplus 0) \rangle$$
 $$=lim_{n\rightarrow\infty}((S^{*}S + Q^{*}Q)x_{n}, x_{n}).$$ which contradicts our hypotheses and the proof is complete.
\end{proof}
\begin{theorem} Let $N$ be a normal operator and let $A$ be an operator in
$B(H)$. If the pair ($A, N$) (resp. ($N, A$) has the property $(FP)_{B(H)}$, then the
Equations $NX - XA = C$ (respect. $AX - XN = C$)
have a solution $X$ if and only if
$$ \begin{pmatrix} N& 0 \\
                            0 & A\end{pmatrix}\,\text{and}\, \begin{pmatrix} N& C \\
                            0& A\end{pmatrix}$$
$$(\text{resp.}\, \begin{pmatrix} A& 0 \\
                            0 & N\end{pmatrix}\,\text{and}\, \begin{pmatrix} A& C \\
                            0& N\end{pmatrix}$$
are similar operators on $H\oplus H$.
\end{theorem}
\begin{proof} If the equation
$NX - XA = C$ has a solution $X$, then
$$ \begin{pmatrix} I& -X \\
                            0& I\end{pmatrix}\begin{pmatrix} N& 0 \\
                            0& A\end{pmatrix}\begin{pmatrix} I& X\\
                            0 & I\end{pmatrix}=\begin{pmatrix} N& NX - XA \\
                            0 & A\end{pmatrix}=\begin{pmatrix} N& C \\
                            0 & A\end{pmatrix}.$$
Therefore
$$ \begin{pmatrix} N& 0 \\
                            0 & A\end{pmatrix}\,\text{and}\, \begin{pmatrix} N& C \\
                            0& A\end{pmatrix}$$
$$(\text{resp.}\, \begin{pmatrix} N& 0 \\
                            0 & A\end{pmatrix}\,\text{and}\, \begin{pmatrix} N& C \\
                            0& A\end{pmatrix}$$
are similar operators on $H\oplus H$. Conversely, if
$$ \begin{pmatrix} N& 0 \\
                            0 & A\end{pmatrix}\,\text{and}\, \begin{pmatrix} N& C \\
                            0& A\end{pmatrix}$$
$$(\text{resp.}\, \begin{pmatrix} N& 0 \\
                            0 & A\end{pmatrix}\,\text{and}\, \begin{pmatrix} N& C \\
                            0& A\end{pmatrix}$$
are similar operators on $H\oplus H$, then there exists an invertible matrix operator

$$\begin{pmatrix} Q& R\\
                            S& T\end{pmatrix}$$ on $B(H\oplus H)$ such that
$$\begin{pmatrix} N& 0 \\
                            0 & A\end{pmatrix}\begin{pmatrix} Q& R\\
                            S & T\end{pmatrix}=\begin{pmatrix} Q& R \\
                            S & T\end{pmatrix}\begin{pmatrix} N& C\\
                            0 & A\end{pmatrix}.$$
Hence
$$QN = NQ, NR - RA = QC, AS = SN, AT - TA = SC.$$
By applying the property $(FP)_{B(H)}$, we obtain
$$AS^{*} = S^{*}N,\,  NQ^{*} = QN^{*}.$$
Therefore
$$NS^{*}S = S^{*}SN,$$
that is, $N$ commutes with $S^{*}S$ and $T^{*}T$. Furthermore, we have
$$(S^{*}S + Q^{*}Q)C = Q^{*} (NR - RA) + S^{*} (AT- TA)$$
$$= (NQ^{*}R + NS^{*}T) - (Q^{*}RA + S^{*}TA)$$
$$= N(Q^{*}R + S^{*}T) -(Q^{*}R + S^{*}T)A.$$

By Lemma 2.1 the operator
$$S^{*}S + Q^{*}Q$$
is invertible and commute with $N$. Hence
$$NX-XA = C,$$
$$X = (S^{*}S + Q^{*}Q)^{-1}(Q^{*}R + S^{*}T).$$
If the pair $(N, A)$ satisfies the $(FP)_{B(H)}$ property, then the equation
$$AX - XN = C$$
has a solution $X$ given by
$$X = -(QS^{*} + RT^{*})(SS^{*} + TT^{*})^{-1}.$$
\end{proof}
\begin{corollary} Let $N, A$ be two operators in $B(H)$ with $A$ a normal operator. If the
pair $(A, N)$ (resp. $(N, A)$) has the $(FP)_{B(H)}$ property, then
$$R(\delta_{A, N} ) = \{ \begin{pmatrix} N& 0 \\
                            0 & A \end{pmatrix}\, \text{and}\, \begin{pmatrix} N& C \\
                            0 & A \end{pmatrix}\, \text{are similar}\}$$
Respectively,
$$R(\delta_{N,A}) =\{\{ \begin{pmatrix} A& 0 \\
                            0 & N \end{pmatrix}\, \text{and}\, \begin{pmatrix} A& C \\
                            0 & N \end{pmatrix}\, \text{are similar}\},$$
where
$\delta_{A, B}$ is the generalized derivation defined on $B(H)$ by
$\delta_{A, B}(X) = AX - XB$.
\end{corollary}
\begin{theorem} Let $\Gamma_{\mathcal{I}}$ be the collection of pairs of operators $(A,B)$ satisfying
the $(FP)_{B(H)}$ property. Then the following assertions are equivalent:\\
(i) $(R, S)\in\Gamma_{\mathcal{I}}$ if $R$ and $S$ are unitary equivalent to $A$ and $B$ respectively.
(ii) $(B, A) \in\Gamma_{I}$.\\
(iii) $(A^{-1}, B^{-1}) \in\Gamma_{\mathcal{I}}$ if $A$ and $B$ are invertible.\\
(iv) $(\lambda A, \lambda B) \in\Gamma_{\mathcal{I}}$ for all  $\lambda\in \mathbb{C}$.\\
(v) $(\lambda I + A, \lambda I + B) \in\Gamma_{\mathcal{I}}$  for all  $\lambda\in\mathbb{C}$.
\end{theorem}
\begin{proof}  i) Assume that $R$ and $S$ are unitary equivalent to $A$ and $B$
respectively. Then there exist two unitary operators $U$ and $V$ such that
$R = UAU^{*}$ and $S = V BV^{*}$ .
If
$RX = XS$, for $X\in B(H)$, then $AU^{*}XV = U^{*}XV B$.
Now since $(A,B)\in\Gamma_{\mathcal{I}}$, it results that $U^{*}XV\in B(H)$. Therefore,
$$A^{*}U^{*}XV = U^{*}XV B^{*}.$$
By this we obtain
$$UA^{*}U^{*}X = XV B^{*}V^{*},$$
from where
$R^{*}X = XS^{*}$. Which proves that $(R, S)\in \Gamma_{\mathcal{I}}$.\\
(ii) If $B^{*}X = XA^{*}$ for $X\in B(H)$, then $AX^{*}= X^{*}B$ and since $X^{*}\in B(H)$,
$A^{*}X^{*} = X^{*}B^{*}$, that is, $XA = BX$.\\
(iii) If
$A^{-1}X = XB^{-1}$, for $X\in B(H)$, then
$A(A^{-1}X)B = A(XB^{-1})B$,
that is, $AX = XB$ and so, $A^{*}X = XB^{*}$. Hence
$$(A^{*})^{-1}A^{*}X(B^{*})^{-1} = (A^{*})^{-1}XB^{*} (B^{*})^{-1},$$
therefore $(A^{*})^{-1}X = X(B^{*})^{-1}.$\\
(iv) If $(\lambda A)X = X(\lambda B)$, for $X\in B(H)$,
then $AX = XB$ and hence $A^{*}X = XB^{*}$.
Consequently,
$\overline{\lambda} A^{*}X = X\overline{\lambda}B^{*}$.\\
(v ) If $(A + \lambda I)X = X(B + \lambda I)$,
then $AX = XB$. Therefore, $A^{*}X = XB^{*}$
and hence
$$(A + \lambda I)^{*}X = X(B +\lambda  I)^{*}.$$
\end{proof}
The $(FP)_{B(H)}$ property hypothesis can be weakened as in the following
corollary.
\begin{corollary} Let
$\Omega_{\mathcal{I}}$ be the collection of pairs of operators $(A,N), (N,A)$
for which $NX - XA = C$ (resp. $AX - XN =C$) have solutions $X$ . Assume
that $(A,N)\in\Omega_{\mathcal{I}}$ (resp. $(N,A)\in \Omega_{\mathcal{I}} $, then\\
(i) $(R, S) \in \Omega_{\mathcal{I}}$  (resp. $(S,R)\in \Omega_{\mathcal{I}}$) if $R$ and $S$ (resp. $S$ and $R$ are unitary equivalent to $A$ and $N$ (resp. to $N$ and $A$).\\
(ii) $(N^{*}, A^{*})\in \Omega_{\mathcal{I}}$
(resp. $(A^{*}, N^{*})\in\Omega_{\mathcal{I}}$ )\\
(iii) $(A^{-1}, N^{-1})\in \Omega_{\mathcal{I}}$  (resp. $(N^{-1}, A^{-1}) \in\Omega_{\mathcal{I}}$ ) if $A$ and $N$ are invertible.\\
(iv) $(\lambda A, \lambda N)\in \Omega_{\mathcal{I}}$  for all $\lambda\in\mathbb{C}$ (resp. $(\lambda N, \lambda A)\in \Omega_{\mathcal{I}}$  for all  $\lambda\in \mathbb{C})$.\\
(v) $(\lambda I + A, \lambda I + N)\in \Omega_{\mathcal{I}}$  for all  $\lambda\in\mathbb{C}$ (resp. $(\lambda I + A, \lambda I + N)\in\Omega_{\mathcal{I}}$  for all  $\lambda\in\mathbb{C}$) .
\end{corollary}
For any operator $A$ in $B(H)$ set, as usual, [$A^{*}, A$] = $A^{*}A - AA^{*}$ (the
self commutator of $A$), and consider the following standard definitions: \\
$A$ is hyponormal if  [$A^{*}, A$] is nonnegative, normal if $A^{*}A = AA^{*}$, subnormal if
it admits a normal extension. An operator $A\in B(H)$ is called dominant by
J.G.Stampfli and B.L.Wadhwa [15] [6] if, for all complex , range($A-\lambda I)
\subseteq \text{range}(A - \lambda I)^{*}$, or equivalently, if there is a real number $M_{\lambda}\geq 1$ such that
$||(A - \lambda I)^{*}f||\leq M_{\lambda} ||(A -\lambda I)f||$, for all $f \in H$. If there exists a real number $M$ such that $M_{\lambda} \leq M$ for all $\lambda$ , the dominant operator $A$ is said to be $M$-hyponormal. A 1-hyponormal is hyponormal. An operator $A$ is
said to be $p$-hyponormal if (for some $0 < p\leq 1$ $(A^{*}A)^{2p}\geq(AA^{*})^{2p}$, $k$-quasihyponormal if $A^{*k}(A^{*}A - AA^{*})A^{k}\geq 0$ ($k\in\mathbb{N})$. If $k = 1$, $A$ is said to be quasi-hyponormal.
Let ($N$), ($SN$), ($\mathcal{H}$), ($p -\mathcal{H}$), ($D$), $Q$, $Q(k)$ denote the classes constituting of normal, subnormal, hyponormal, p-hyponormal operators, dominant,
quasi-hyponormal and k-quasihyponormal operators. Then
$$(\text{N}) \subset (\text{SN}) \subset (\mathcal{H})\subset (\text{m} -\mathcal{H})\subset (\text{D})$$
and
$$(\text{N}) \subset (\text{SN})\subset  (\mathcal{H})\subset (p -\mathcal{H}).$$
\begin{corollary} Let $N$ be a normal operator and let $A$ be an operator in
$B(H)$. If the pair ($A, N$) (resp. ($N, A$) has the property $(FP)_{B(H)}$, then the
equations
$NX - XA = C$ ( resp. $AX - XN = C$)
have a solution X if and only if
$$ \begin{pmatrix} N& 0 \\
                            0 & A \end{pmatrix}\, \text{and}\, \begin{pmatrix} N& C \\
                            0 & A \end{pmatrix}$$
and $$(\text{resp.} \begin{pmatrix} A& 0 \\
                            0 & N \end{pmatrix}\, \text{and}\, \begin{pmatrix} A& C \\
                            0 & N \end{pmatrix})$$
are similar operators on $H \oplus H$ under either of the following cases:\\
(i) $A$ dominant.\\
(ii) $A$ $p$-hyponormal.\\
(ii) $A$ $k$- quasihyponormal.
\end{corollary}
\begin{proof} It is well known [15], [2] that the pair ($N, A$) (resp. ($A, N$) has
the $(FP)_{B(H)}$ property under either of the above cases.
\end{proof}
\section{Application}
\begin{theorem}
 If the $\sigma(A)$ and $\sigma(B)$ are disjoint, then $AX-XB=Y$ has a unique solution given by:
            $$X= \frac{1}{2\pi i}\int_{\Gamma}(\lambda-A)^{-1}Y(\lambda-B)^{-1}d\lambda,$$
            where $\Gamma$ is a Cauchy contour which separate the spectrum of $A$ and the spectrum of $B$.
 \end{theorem}
 \begin{proof}
 We consider the equation $AXB-CXD=E$. If the spectrum $\sigma(C, A)$ and the spectrum $\sigma(B,D)$ are disjoint,
            then
            $$X= \frac{1}{2\pi i}\int_{\Gamma}(\lambda C-A)^{-1}Y(\lambda B-D)^{-1}d\lambda,\eqno(*)$$
            where $\Gamma$ is a Cauchy contour which separate the spectrum of $\lambda C-A$ and the spectrum of $\lambda B-D$, where
            $\sigma(C, A)$ is  the set of $\lambda\in\mathbb{C}$ such that $\lambda C- A$ is not invertible.
        \end{proof}

		\begin{theorem}
			Let $ T_{1} : X_{1}\rightarrow X_{2}$ and $T_{2} : X_{1} \rightarrow X_{2}$  two linear operators between Banach spaces and their spectrum in the open unit disque. Then, if $Y \in  L(X_{2};X_{1})$, the equation
			$$T_{1}XT_{2} -X = Y $$
			Has a unique solution  $X \in  L(X_{2};X_{1})$, given by
			$$X =-\sum^{\infty}_{k=0} T_1^{k}YT_2^{k}\eqno(1)$$
		
	\end{theorem}
		\begin{proof}
			We apply the theorem with  $\lambda C-A = \lambda I_{1}-T_{1}$ and $\lambda B-D = \lambda T_{2}- I_{2}$ .
			where $I_{k}$ is the identity operator on $X_{k}(k = 1; 2)$. Clearly,
			$\sigma(C;A) =\sigma(T_{1})$ and $\sigma(B;D) = \{\lambda^{-1}: \lambda\in\sigma (T_{2})\}$.
			Since $\sigma (T_{1})$ and  $\sigma(T_{2})$ are in the open unit disk, the spectrum $\sigma (C;A)$ and  $\sigma(B;D)$ are disjoint, and the precedent theorem  imply that our equation has a unique solution given by

			$$X =\frac{1}{2\pi i}\int_{\omega}(\lambda I_{1}-T_{1})^{-1}Y(\lambda T_{2}- I_{2})d\lambda \eqno(2).$$
			Here $\omega$ is the unit circle  with anti-clockwise orientation. We have
			$$(\lambda I_{1}-T_{1})^{-1}=\sum^{\infty}_{k=1}\frac{1}{\lambda ^{k}} T^{k-1}_{1}$$
				$$(\lambda T_{2}-I_{2})^{-1}= -\sum^{\infty}_{k=0}\lambda ^{k}T_{2}^{k}.$$
				Since $ \sigma(T_{1})$ and $\sigma(T_{2})$ are in the open unit disque.\\ It results that the right sides of (1) and (2) are equal.
			\end{proof}
	It is well known $\sigma(A)\cap\sigma(B)=\emptyset$, then  the equation
$AX-XB = Y$ has a unique solution for all $ Y$. Which comes from the fact that for an operator
 $\delta_{A,B} : B(H) \rightarrow B(H)$, defined by  $\delta_{A,B}(X) =AX-XB$,  we have il $\sigma(\delta_{A, B}) \subset  \sigma(A) + \sigma(B)$.
  It is also possible to find an explicit formula of this solution with respect to $ Y$.
		
		\begin{theorem}
			1)	If $\sigma(A)\subset \{z: Re(z) > 0\}$ and $\sigma(B) \subset \{z: Re(z)<0\}$, then
			$$X= \int^{\infty}_{0}e^{-tA}Ye^{tB} dt.$$
			2)	If $\sigma(A)\subset \{z: |z|>\rho\}$ and $\sigma(B) \subset \{z: |z|< \rho\}$, then $$ X= \sum^{\infty}_{n=0}A^{-n-1}YB^{n}.$$
		\end{theorem}

		\begin{proof}
			Since the spectrum $\sigma(A)$ and $\sigma(B)$ are compact sets, hence $\sigma(A)$ and $\sigma(B)$ can be separated by a strip of positive width $d$, in the first case and in the second case $\sigma(A)$ is contained in a disk of radius $\rho_{1}$ and $\sigma(B)$ is contained in the compliment of a disk of radius $\rho_{2}$. We deduce that for all $t\geq 0$, $||e^{-tA}Ye^{tB}||\leq e^{-td}$
			and for all $n\geq 0$, $||A^{-n-1}YB^{n}||\leq
			(\frac{\rho_{1}}{\rho_{2}})^{n}$. It results that the series and the integral converge.
			Furthmore,
	
			1.
$$	AX-XB=\int^{\infty}_{0}Ae^{-tA}Ye^{tB} -\int^{\infty}_{0}e^{-tA}Ye^{tB}Bdt$$
			$$= \int^{\infty}_{0}\frac{d}{dt}(-e^-{tA}Ye^{tB})dt=Y$$
		2.
$$  AX-XB=\sum^{\infty}_{n=0}A^{-n}YB^{n}-\sum^{\infty}_{n=0}A^{-n-1}YB^{n+1}=Y$$
	\end{proof}
	
        \subsection{Equation of Monkeypox Diseases}
            Is an equation of the form $ A^{*}XA+ tAXA = Y,$
            where $A$  represented a known influence defined in the Hilbert space $H$ and
            represents the causal and auxiliary factors of Smallpox (factor
            hormonal, nutritional factor, hereditary factor and virus …), and $Y$
            is also a well-known effect defined in the Hilbert space H and
            represents the relationships of factors that contribute to Smallpox (variole)
            (Sex, age, exposure to radiation, use of sedatives, psychological stress, factors that weaken immunity, etc.)
        $t$ is a positive real number called control factor and
            the unknown influencer $X$. This influencer represents the treatment in
            terms of quantity and quality for the patient.

            For the solution we take in $(*)$, $A=A^{*},B=A, C=D=A$. If $A=I$, we have $X= \frac{1}{t+1}Y$.
            if $\sigma(\lambda tA-A^{*})\cap\sigma(\lambda A-A))=\emptyset$,
           $$X= \frac{1}{2\pi i}\int_{\Gamma}(\lambda tA-A^{*})^{-1}Y(\lambda A-A)^{-1}d\lambda,\eqno(*)$$
\section{Some Problems}
The operator $A \in B(H)$ is said to be finite [16] if $||I - (AX - XA)|| \geq 1$ (*) for all $X\in B(H)$, where $ I$ is the identity operator. The well-known
inequality (*), due to [16] is the starting point of the topic of commutator
approximation (a Topic which has its roots in quantum theory [17]). This
topic deals with minimizing the distance, measured by some norm or other,
between a varying commutator (or self-commutator $XX^{*}-X^{*}X$) and some
fixed operator [1, 6, 8] we begin by the definition of the best approximant of an operator. Let
$E$ be a normed space and $M$ a subspace of $E$. If to each $A \in E$ there exists
an operator $B \in M$ for which $||A - B||\leq||A - C||$ for all $C \in M$.
Such $B$ (if they exist) are called best approximants to $A$ from $M$. To
approach the concept of an approximant consider a set of mathematical
objects (complex numbers, matrices or linear operator, say) each of which
is, in some sense," nice", i.e. has some nice property P (being real or self-adjoint,
say): and let $A$ be some given, not nice, mathematical object: then
a $P$ best approximants of $A$ is a nice mathematical object that is" nearest
" to $A$. Equivalently, a best approximant minimizes the distance between
the set of nice mathematical objects and the given, not nice object.
Of course, the terms" mathematical object"," nice"," nearest", vary
from context to context. For a concrete example, let the set of mathematical
objects be the complex numbers, let" nice" = real and let the distance be
measured by the modulus, then the real approximant of the complex number
$z$ is the real part of it, $Rez = {(z+\overline{z})\over{2}}$.
2 .  Thus, for all real $x$
$$|z - Rez|\leq |z - x|. $$
\subsection{ Problem I}
We define the von Neumann-Schatten classes $C_{p}$ ($1\leq p < \infty$).
Let $B(H)$ denote the algebra of all bounded linear operators on a complex separable Hilbert space $H$ and let $ T \in B(H)$ be compact, and
let $s_{1}(T) \geq s_{2}(T) \geq\cdots \geq 0$ denote the singular values of $T$, i.e., the
eigenvalues of $|T| = (T^{*}T)^{1\over 2}$ arranged in their decreasing order. The operator $T$ is said to be belong to the Schatten $p$-classes $C_{p}$ if
$$||T||_{p} =[\sum_{i=0}^{\infty}s_{i}(T)^{p}]^{1\over p}=[tr(T)^{p}]^{1\over p}, 1\leq p<\infty$$
where tr denotes the trace functional. Hence $C_{1}$ is the trace class, $C_{2}$
is the Hilbert-Schmidt class, and $C_{\infty}$ is the class of compact operators
with $||T||_{\infty}= s_{1}(T)= sup\{||Tx||:\, ||x||=1\}$ denoting the usual operator norm.

The related topic of approximation by commutators $AX-XA$ or by generalized
commutator $AX-XB$, which has attracted much interest, has its
roots in quantum theory. The Heisenberg Uncertainly principle may be
mathematically formulated as saying that there exists a pair $A, X$ of linear
transformations and a non-zero scalar $\alpha$ for which
$$AX - XA = \alpha I \eqno (3.1)$$
Why to solve the operator equation $AX-XB = C.$
Clearly, (3.1) cannot hold for square matrices $A$ and $X$ and for bounded
linear operators $A$ and $X$. This prompts the question:
how close can $AX - XA$ be the identity?
Williams [17] proved that if $A$ is normal, then, for all $X in B(H)$,
$$||I-(AX - XA) ||\geq||I||. \eqno(3.2)$$
Mecheri [7] generalized Williams inequality (3.2): he proved that if $A, B$
are normal, then for all $X \in B(H)$
$$||I -(AX - XB)||\geq ||I||. \eqno(3.3)$$
Anderson [1] generalized Williams inequality (3.2): he proved that if $A$
is normal and commutes with $B$ then, for all $X\in B(H)$
$$||B -(AX - XA)|| \geq||B||. \eqno(3.4)$$
Maher [6] obtained the $C_{p}$ variants of Anderson's result. Mecheri [8] studied
approximation by generalized commutators $AX-XC$: he showed that the
following inequality holds
$$||B -(AX - XC)||_{p}  \geq||B||_{p}.\eqno (3.5)$$
for all $X \in C_{p}$ if and only if $B\in ker\delta_{A,B}$.
In the above inequalities (3.2),(3.3), (3.4) and (3.5) the zero commutator
is a commutator approximant in $C_{p}$ of $B$.
\subsection{Problem II}
Let $\delta_{A}$ be the operator defined on $B(H)$ by
$\delta_{A}(X) = AX - XA$. It is known that I is not commutator, that is, $ I\not\in R( A)$. Anderson [1] proved that there exists $A\in B(H)$ such that $I \in R(A)$, that is, the distance from $I$ to $AX - XA$ is minimal, i.e., equal to zero. For more details see Mecheri [7]
In [8] We constructed a pair $(A, X)$ of elements in $B(H)$ such that
$dist(I,R(A)) < 1.$
Open question: Does $ dist(I,R(A)) = r \in (0, 1)$ implies for all invertible
$S$ that $dist(I,R(SAS^{-1})) = r \in (0, 1)$.
\subsection{Problem III}
Let $J_{A}(H) = \{A \in B(H) :\, I \in \overline{R(A)}\}.$
Here is a problem that might of interest. Recall [5] if $T : X \rightarrow Y$ define
$$T = \{lim_{n}Tx_{n} : sup_{n}||x_{n}|| < 1\},$$
where $X, Y$ are Banach spaces differing from the usual closure in that its
points have to be the limits of images of bounded sequences of vectors so:\\
\textbf{Question}. For which operators  $A$ on Hilbert space $H$ do we have
$I \in \overline{R(A)}$ ?

\end{document}